%% file: main.tex
\documentclass[11pt]{amsart}
\usepackage[utf8]{inputenc}

\usepackage{import}
\usepackage{xcolor}

\include{preamble}

\title{
Lyapunov exponents for truncated unitary and Ginibre matrices
}
\author{Andrew Ahn and Roger Van Peski}
\date{\today}

\begin{document}

\maketitle

\begin{abstract}
In this note, we show that the Lyapunov exponents of mixed products of random truncated Haar unitary and complex Ginibre matrices are asymptotically given by equally spaced `picket-fence' statistics. We discuss how these statistics should originate from the connection between random matrix products and multiplicative Brownian motion on $\GL_n(\C)$, analogous to the connection between discrete random walks and ordinary Brownian motion. Our methods are based on contour integral formulas for products of classical matrix ensembles from integrable probability. \\

\noindent
\emph{AMS 2020 Mathematics Subject Classification: 15B52, 60B20} \\

\noindent
\emph{Keywords: random matrix products, Lyapunov exponents, picket fence statistics}
\end{abstract}

\section{Introduction}

\subsection{Model and main result.} Suppose $X_1,X_2,\ldots$ is a sequence of $n\times n$ random matrices and let $\sigma_1(T),\ldots,\sigma_n(T)$ be the singular values of the product $X_1 \cdots X_T$. By the multiplicative ergodic theorem of Oseledec\footnote{Related existence results on limits of matrix entries were obtained earlier by Bellman \cite{bellman1954limit} and Furstenberg-Kesten \cite{furstenberg1960products}.} \cite{oseledets1968multiplicative}, when the matrices $X_\tau$ are iid, the limits 
\[
\mu_i = \lim_{T \to \infty} \frac{\log \sigma_i(T)}{T}
\]
exist under general assumptions, see also \cite{raghunathan1979proof}. The constants $\mu_i$ are known as \emph{Lyapunov exponents} due to the interpretation of $X_\tau$ as a transfer matrix in the evolution of a chaotic discrete-time dynamical system \cite{crisanti1993products}. The behavior of the Lyapunov exponents for matrix products as $T$ tends to infinity, together with the related question of the fluctuations of log singular values, are now the subject of a large literature with connections to chaotic dynamical systems and disordered systems in statistical physics, neural networks, and other areas; we will not attempt to survey it here, but mention \cites{akemann2019integrable,crisanti1993products,gorin2018gaussian,hanin2020products} and the references therein.

A natural question, given that the Lyapunov exponents exist under general conditions, is whether there are any universality results about their values which are stable under different choices of random matrices $X_\tau$, at least in the large $n$ limit. When $X_\tau$ are iid $n \times n$ complex Gaussian (Ginibre) matrices, the rescaled logarithms of the squared singular values of $X_1 \cdots X_T$ were shown in \cite{akemann2019integrable} to converge to $0,-1,-2,\ldots$, in the regime where $T$ and $n$ tend to infinity with $T/n \to \infty$; they refer to these as `picket fence statistics'. This suggested that similar results should hold in the analogous setting of the double limit $T \to \infty$, then $n \to \infty$, i.e. the large $n$ limits of Lyapunov exponents. The regime where $T/n$ converges to a constant was also studied in \cites{akemann2019integrable,akemann2020universality,liu2018lyapunov}. Though the Lyapunov exponents do not directly appear in this regime, it is known from the work of the first author \cite{ahn2019fluctuations} that the asymptotic behavior of the largest singular values for products of Ginibre matrices then coincides with that of products of corners of Haar unitary matrices. This made it natural to expect that the sequence of Lyapunov exponents of both Ginibre and truncated unitary corners should converge to the same limit $0,-1,-2,\ldots$. Another motivation to investigate this question came from work of the second author \cite[Theorem 1.2]{van2020limits}, in which the analogues of Lyapunov exponents for $p$-adic random matrices were found to converge to a similar limit--a geometric progression rather than an arithmetic one--as $n \to \infty$ for the $p$-adic analogues of Ginibre and truncated unitary matrices.

Our main result, \Cref{thm:lyapunov}, tells that in the large $n$ limit any mixture of Ginibre and truncated Haar matrices yields the same picket fence statistics. Namely, if the matrices $X_\tau, \tau \geq 1$ are each distributed as either Ginibre matrices or corners of Haar matrices from some unitary groups $U(L_\tau), L_\tau > n$, independent but not necessarily iid, then under weak technical assumptions the picket fence statistics $0,-1,-2,\ldots$ appear in the large $n$ limit. 

\begin{definition}\label{def:Pnl}
Given $L \in \Z_{>n}$, we say that an $n\times n$ random matrix $X$ is $\P_{n,L}$-distributed if $X = \sqrt{L} \wt{X}$ where $\wt{X}$ is an $n\times n$ submatrix of an $L\times L$ Haar-distributed unitary matrix. We say that $X$ is $\P_{n,\infty}$-distributed if $X$ is an $n\times n$ complex Ginibre matrix, i.e. its entries are iid standard complex Gaussians.
\end{definition}

\begin{theorem}\label{thm:main_intro}
For each $n > 0$, let $(L_\tau^{(n)})_{\tau \geq 1}$ be a sequence with $L_\tau^{(n)} > n$ for all $\tau$, such that the limiting frequencies
\[ \lim_{T\to\infty}\frac{ \#\{ L_\tau^{(n)} - n \ge k: 1 \leq \tau \leq T\}}{T} \]
exist for every $k$. Suppose $X_1^{(n)},X_2^{(n)},\ldots$ is a sequence of independent random matrices such that $X_\tau^{(n)}\sim \P_{n,L_\tau^{(n)}}$, and let $y_1^{(n)}(T) \ge \ldots \ge y_n^{(n)}(T)$ denote the squared singular values of the product $X_T^{(n)} \cdots X_1^{(n)}$. Then
\[
 \frac{ \log y_i^{(n)}(T)}{T} \to \lambda_i(n) \quad \quad \text{ in probability as $T \to \infty$} 
\]
for explicit $\lambda_i(n)$ given in \Cref{thm:concentration}. Furthermore, these Lyapunov exponents converge to picket fence statistics
\[
c(n)^{-1}(\lambda_i(n) - \lambda_1(n)) \to -i + 1 \quad \quad \quad \text{ as }n \to \infty
\]
for each $i = 1,2,\ldots$, where $c(n) > 0$ is given explicitly in \Cref{thm:lyapunov}.
\end{theorem}

The constant $\lambda_i(n)$ is twice the $i\tth$ Lyapunov exponent, but we found this normalization more convenient. We compute $\lambda_i(n)$ exactly using contour integral formulas for moments of the singular values, obtaining formulas which are suitable for taking the large $n$ limit. The contour integral formulas rely on results from \cite{ahn2019fluctuations}, which arise due to connections between matrix products and Macdonald processes \cite{borodin2014macdonald}, see also \cites{borodin2015general,borodin2018product,gorin2020crystallization}. Similar techniques were used to analyze fluctuations of log singular values of matrix products in \cite{gorin2018gaussian}. As a side note, the special case of these computations when $X_\tau^{(n)}$ are iid gives a simple expression for the Lyapunov exponents of truncated unitary matrices, which we are not aware of in the literature, and also recovers expressions for the Lyapunov exponents of Ginibre matrices obtained in \cite{forrester2013lyapunov}. See \Cref{ex:ginibre} for details.

\subsection{Connection with Brownian motion on $\GL_n(\C)$.} While our main results are restricted to the case of mixed Ginibre/truncated unitary products, we believe that the picket fence $0,-1,-2,\ldots$ should appear much more universally in large $n$ limits of Lyapunov exponents. In fact, picket fence statistics have already appeared in another related context: the \emph{multiplicative Brownian motion $\mathsf{Y}(t)$ on $\GL_n(\C)$}. This is just the Brownian motion on the real manifold $\GL_n(\C)$ with infinitesimal generator given by the ($1/2$) Laplacian, which may be written as
\[
\frac{1}{2}\sum_{1 \leq \ell,j \leq n} \partial_{x_{\ell j}}^2 + \partial_{i x_{\ell j}}^2
\]
where $\partial_{x_{\ell j}},\partial_{i x_{\ell j}}$ are the left-invariant vector fields associated with the Lie algebra elements\footnote{One obtains the same operator by replacing the set of $E_{\ell j},i E_{\ell j} $ with any orthonormal basis for $\mathfrak{gl}_n(\C)$ with respect to the inner product $\langle A, B \rangle = \Re( \Tr(A^*B))$; we have chosen a specific one only for concreteness.} $E_{\ell j},i E_{\ell j} \in \mathfrak{gl}_n(\C)$. We refer to \cite[Chapter 3]{hsu2002stochastic} for general background on Brownian motion on Riemannian manifolds. 

One has equality of multi-time distributions of the following two stochastic processes on $\R^n$, see \cite[Corollary 3.3]{jones2006weyl}\footnote{As stated in \cite{jones2006weyl}, Corollary 3.3 states that it is the vector of logarithms of \emph{squared} singular values which matches \eqref{item:drift_BM}, causing their statement to differ by a factor of two from ours. The statement as we have reproduced it here follows from Proposition 3.1 earlier in the same paper, and we thank Neil O'Connell for confirming that this normalization is indeed the correct one.}, :
\begin{enumerate}[(I)]
    \item $(\mathsf{y}_1(t),\ldots,\mathsf{y}_n(t))$, where $\mathsf{y}_i(t)$ is the logarithm of the $i\tth$ largest singular value of the multiplicative Brownian motion $\mathsf{Y}(t)$. \label{item:mult_BM}
    \item $(B^{(1)}_t,\ldots,B^{(n)}_t)$, a standard Brownian motion on $\R^n$ started at the origin with drift $(n-1,n-3,\ldots,-n+3,-n+1)$, which is conditioned to remain in the positive Weyl chamber $(x_1 > x_2 > \cdots > x_n)$ for all time.\label{item:drift_BM}
\end{enumerate}
From the above it follows that the log singular values of multiplicative Brownian motion exhibit the same picket fence statistics, 
\[ \lim_{t \to \infty} \frac{1}{2t} \mathsf{y}_i(t) - \tfrac{n-1}{2} = -i + 1. \]
We mention that the drift vector appearing is exactly the sum of the positive roots of $\mathfrak{sl}_n$, and analogous results hold for any semisimple Lie algebra \cite{grabiner1999brownian}, for which we recommend the exposition of \cite{jones2006weyl}. This gives an attractive Lie-theoretic interpretation of picket fence statistics. 

Hints that Dyson Brownian motion with drift connects to matrix product processes, as well as multiplicative Brownian motion, have been observed previously at the level of fluctuations. It is noted in \cite{akemann2019integrable} that large products of large random matrices should relate in the limit to the multiplicative Brownian motion $\mathsf{Y}(t)$, which they refer to as a solution of (a case of) the Dorkhov-Mello-Pereyra-Kumar (DMPK) equation --- where the DMPK equation corresponds to the Fokker-Planck equation for the process. They also note a connection between local statistics of products of complex Ginibre matrices and those of Dyson Brownian motion with equally spaced initial conditions; this connection is also remarked upon by \cite{ipsen2016isotropic} and \cite{gorin2018gaussian}, the latter of which explicitly notes that ``it would be very interesting to find a conceptual explanation for this analogy between products of matrices and Dyson Brownian Motion''. 

These observed connections between Dyson Brownian motion with equally spaced initial conditions and matrix products are explained by the equivalence of \eqref{item:mult_BM} and \eqref{item:drift_BM}, in view of the following fact: the process
\begin{equation}\label{eq:bm_transform}
    (W^{(1)}_t,\ldots,W^{(n)}_t), \quad \quad W^{(i)}_t := t B^{(i)}_{1/t},
\end{equation}
with $B^{(i)}_t$ as in \eqref{item:drift_BM}, is a Brownian motion on $\R^n$ started at $(n-1,n-3,\ldots,-n+3,-n+1)$ with zero drift, conditioned to remain in the positive Weyl chamber $(x_1 > x_2 > \cdots > x_n)$ for all time\footnote{This follows from the general fact that if $B_t$ is Brownian motion on $\R^n$ with initial point $\vec{a}$ and drift $\vec{b}$ conditioned to remain in the positive Weyl chamber $(x_1 > x_2 > \cdots > x_n)$, then transforming the process via $t B_{1/t}$ interchanges the initial condition with the drift.}. In particular, interchanging drift and initial conditions does not change the time $1$ marginals, as noted in \cite[Proposition 2.3(b)]{jones2006weyl}. It is however much more natural to state the process-level equality of $(\mathsf{y}_i(t))_{1 \leq i \leq n}$ with $(B^{(i)}_t)_{1 \leq i \leq n}$ than with $(W^{(i)}_t)_{1 \leq i \leq n}$, as the latter requires keeping track of the time-change \eqref{eq:bm_transform}. 

\subsection{Universality?}
It is natural to view the above connections between Brownian motion with evenly spaced initial conditions and matrix product processes as an accidental consequence of a more basic--but still heuristic--fact that matrix product processes act as a kind of discrete random walk approximation to multiplicative Brownian motion. Our results may be viewed as probing this connection at the level of the drifts/law of large numbers, and the ubiquity of picket fence statistics suggests that in this sense the discrete matrix-product random walk approximations to Brownian motion on $\GL_n(\C)$ become exact as $n \to \infty$. This connection is further supported by \cite{ahn2021unpublished}, which was completed after the first version of the present paper. In that work, the fluctuations of the largest log squared singular values of random matrix products to those of the large $n$ limit of Brownian motion on $\GL_n(\C)$ is established for the class of right unitarily invariant matrices. 

We expect that the universality class of matrix product models exhibiting picket fence statistics is much broader than the class of models considered in this paper. Based on \cite{ahn2021unpublished}, we believe that picket fence statistics appear for arbitrary products of right unitarily invariant complex random matrices under weak hypotheses on the number of singular values close to $0$ and $\infty$, though the optimal hypotheses are not clear. In another direction, we also expect universality of picket fence statistics for products of Wishart matrices (centered iid entries with variance $1/n$).

By contrast, the global limit shapes of the Lyapunov exponents are known to be nonuniversal. Namely, rather than considering the limiting spacings as we do, one may study the $n \to \infty$ limits of the empirical measures 
\[
\frac{1}{n}\sum_{i=1}^n \delta_{\mu_i}
\]
associated to the Lyapunov exponents $\mu_i$ of some sequences of $n \times n$ random matrices. The limiting measures are known to be highly dependent on the distribution of the matrices, see \cite{newman1986distribution,newman1986lyapunov}, though we note that the Wishart case is universal and agrees with Ginibre matrices \cite{isopi1992triangle}. These results are not so surprising from our perspective, as it is typical in random matrix theory \cite{erdHos2012universality}, random tilings \cite{aggarwal2019universality,aggarwal2021edge,aggarwal2021gaussian}, and other models that local statistics similar to the ones we study are much more universal than global limit shapes.

\bigskip

\addtocontents{toc}{\protect\setcounter{tocdepth}{1}}
\subsection*{Acknowledgements}  We thank Alexei Borodin and Vadim Gorin for helpful feedback on an earlier draft, Mario Kieburg for a fruitful conversation which provided the initial impetus to write down these results, Neil O'Connell for answering questions regarding \cite{jones2006weyl}, and the anonymous referees for helpful comments. RVP was partially supported by an NSF Graduate Research Fellowship under grant \#$1745302$, and by the NSF FRG grant DMS-1664619.
\addtocontents{toc}{\protect\setcounter{tocdepth}{2}}

\section{Main Results}

\begin{definition}\label{def:freq_conv}
A sequence $(R_\tau)_{\tau \geq 1}$ with $R_\tau \in \N \cup \{\infty\}$ is \emph{frequency-convergent} if the limiting frequencies
\begin{equation}\label{eq:lim_freq}
    \lim_{T\to\infty}\frac{ \#\{ R_\tau \ge k: 1 \leq \tau \leq T\}}{T} 
\end{equation} 
exists for every $k$. The associated \emph{frequency measure} of such a sequence is the discrete probability measure $\rho$ on $\N \cup \{\infty\}$ with $\rho(\llbracket k,\infty\rrbracket)$ given by \eqref{eq:lim_freq}, where $\llbracket k,\infty\rrbracket := \Z_{\geq k} \cup \{\infty\}$.
\end{definition}

\begin{definition}
Let
\[ H_m := \sum_{k=1}^m \frac{1}{k} \]
denote the $m\tth$ harmonic number, and
\[
\gamma := \lim_{m \to \infty} \left( H_m - \log m\right)
\]
denote the Euler-Mascheroni constant.
\end{definition}

\begin{theorem} \label{thm:concentration}
Fix $n > 0$, let $(L_\tau^{(n)})_{\tau \geq 1}$ be a frequency-convergent sequence with $L_\tau^{(n)} > n$ for all $\tau$, and let $\rho_n$ be the frequency measure associated to $(L_\tau^{(n)}-n)_{\tau \geq 1}$. Let $X_1^{(n)},X_2^{(n)},\ldots$ be independent random matrices such that $X_\tau^{(n)} \sim \P_{n,L_\tau^{(n)}}$, and let $y_1^{(n)}(T) \ge \ldots \ge y_n^{(n)}(T)$ denote the squared singular values of $X_T^{(n)} \cdots X_1^{(n)}$. Then their rescaled logarithms converge in probability to explicit constants,
\begin{equation}\label{eq:main_conv_in_prob}
\frac{1}{T}\log y_i^{(n)}(T) \to \log(n) - \sum_{k=1}^\infty \rho_n(\llbracket k,\infty \rrbracket) \left( \log \left( 1 - \frac{1}{k+n} \right) + \frac{1}{k + n - i} \right) =: \lambda_i(n),
\end{equation}
as $T \to \infty$ for $i=1,\ldots,n$.
\end{theorem}

\begin{theorem} \label{thm:lyapunov}
Suppose $\rho_n$ and $\lambda_i(n)$ are defined as in \Cref{thm:concentration}. Set
\[ c(n) := \sum_{k=1}^\infty \frac{\rho_n(\llbracket k,\infty \rrbracket)}{(k+n-1)^2}. \]
If $m = o(\sqrt{n})$, then
\[ c(n)^{-1}(\lambda_i(n) - \lambda_1(n)) = -i+1 + o(1) \]
for $1 \leq i \leq m$ as $n\to\infty$.
\end{theorem}

Note that \Cref{thm:main_intro} is a restatement of \Cref{thm:concentration,thm:lyapunov} without explicit constants.

\begin{remark}\label{rmk:error}
We note that taking $m = o(\sqrt{n})$ is optimal in \Cref{thm:lyapunov}. If we assume $m = o(n)$, we can replace the conclusion with the relative error
\[ c(n)^{-1}(\lambda_i(n) - \lambda_1(n)) = (-i+1)(1 + o(1)) \]
in place of the absolute error. However, if $m = pn$, then even the statement with relative error fails to hold, see \Cref{ex:ginibre} below. 
\end{remark}

\begin{example}\label{ex:ginibre}
The special case of \Cref{thm:concentration} for iid matrices $X_\tau^{(n)}$ is worth mentioning. If $L_\tau^{(n)} = L$ for all $\tau$ and some finite $L$, so that all $X_\tau^{(n)}$ are (normalized) truncated unitary matrices, 
\[
\rho_n(\llbracket k,\infty \rrbracket) = \begin{cases}
1 & k \leq L \\
0 & k > L
\end{cases}
\]
and manipulating \eqref{eq:main_conv_in_prob} shows
\begin{equation}\label{eq:truncated_lyapunov_computation}
    \lambda_i(n) =   H_{n-i} + \log L - H_{L-i}.
\end{equation}
In the limit case of Ginibre matrices $L_\tau^{(n)} = \infty$, one similarly has 
\begin{equation}\label{eq:forrester}
    \lambda_i(n) =  H_{n-i} - \gamma.
\end{equation}
In this case 
\[ c(n) = \sum_{k=1}^\infty \frac{1}{(k+n-1)^2} \]
and it follows from \eqref{eq:lambda_diff} that
\begin{equation}\label{eq:ginibre_c_bounds}
    c(n)^{-1}(\lambda_m(n) - \lambda_1(n)) + (m-1) = -(m-1)^2 c(n)^{-1} \sum_{k=1}^\infty \frac{1}{(k+n-1)^2(k+n-m)}.
\end{equation}
The RHS of \eqref{eq:ginibre_c_bounds} is $O(m^2/n)$ if $m = o(n)$ and is of order $n$ if $m = pn$ for some fixed $p \in (0,1]$, as $n\to\infty$, as mentioned in \Cref{rmk:error}. In particular, from the fact that it is $O(m^2/n)$ when $m = o(n)$ we see that the condition $m = o(\sqrt{n})$ in \Cref{thm:lyapunov} is sharp.
\end{example}

\begin{remark}\label{rmk:forrester}
The formula \eqref{eq:forrester} was previously obtained in the literature in \cite[Proposition 1]{forrester2013lyapunov}, and \eqref{eq:truncated_lyapunov_computation} is a natural generalization of that result. For the reader's benefit in comparing our results with \cite{forrester2013lyapunov} we note that there are a couple of typos in the statements of results in that work: \cite[Proposition 1]{forrester2013lyapunov} is missing a factor of $1/2$ on the RHS, as may be seen by inspecting \cite[(2.10)]{forrester2013lyapunov} in its proof. The factor of $1/2$ does not appear in our \eqref{eq:forrester} because we consider the squared singular values of the matrix product rather than its singular values, so our $\lambda_i(n)$ are twice the associated Lyapunov exponent. We note also that the factor of $1/2$ is present in the more general \cite[Proposition 2]{forrester2013lyapunov} and \cite[Corollary 1]{forrester2013lyapunov}, but in those results the arguments of the digamma function should be $d-k+1$ and $d-m+1$ rather than $k$ and $m$, respectively.
\end{remark}

\section{Proofs}

Throughout this section, for each integer $n > 0$ let $X_1^{(n)},X_2^{(n)},\ldots$ be independent random matrices such that $X_\tau^{(n)} \sim \P_{n,L_\tau^{(n)}}$ for some sequence $(L_\tau^{(n)})_{\tau \ge 1}$ of positive integers. Let $y_1^{(n)}(T) \ge \cdots \ge y_n^{(n)}(T)$ denote the squared singular values of $X_T^{(n)} \cdots X_1^{(n)}$. 

We rely on contour integrable expressions for the joint moments of the empirical measures $\tfrac{1}{n} \sum_{i=1}^n \delta_{y_j^{(n)}(T)}$ obtained in \cite{ahn2019fluctuations}:

\begin{proposition}[{\cite[Theorem 4.3]{ahn2019fluctuations}}]\label{thm:moment_formula}
If $c_1,\ldots,c_m > 0$ are real, and $n < L_1^{(n)},\ldots, L_T^{(n)} < \infty$ are positive integers, then
\begin{align} \label{eq:moment_formula}
\begin{multlined}
\E\left[ \prod_{i=1}^m \sum_{j=1}^n \left(y_j^{(n)}(T)\right)^{c_i} \right] = \frac{\prod_{i=1}^m (-c_i)^{-1}}{(2\pi\bi)^m} \oint \cdots \oint \prod_{1 \le i < j \le m} \frac{(u_j - u_i)(u_j + c_j - u_i - c_i)}{(u_j - u_i - c_i)(u_j + c_j - u_i)} \\
\times \prod_{i=1}^m \left( \prod_{\ell=1}^n \frac{u_i + \ell - 1}{u_i + c_i + \ell - 1} \cdot \prod_{\tau=1}^T (L_\tau^{(n)})^{c_i} \prod_{k=1}^{L_\tau^{(n)} - n} \frac{u_i + c_i - k}{u_i - k} \right) du_i
\end{multlined}
\end{align}
where the $u_i$-contour $\fU_i$ is positively oriented around $\{-c_i- \ell+1\}_{\ell=1}^n$ but does not enclose $k$ for $1 \le k \le \max_{1 \le \tau \le T} L_\tau^{(n)} - n$, and is enclosed by $\fU_j - c_i, \fU_j + c_j$ for $j > i$; the result holds provided that such contours exist.
\end{proposition}

\begin{remark}
\Cref{thm:moment_formula} is a restatement of \cite[Thm. 4.3]{ahn2019fluctuations} in terms of the parameters used in this paper (see \cite[Thm. A.1]{ahn2019fluctuations} to translate from the parameters used in \cite[Thm. 4.3]{ahn2019fluctuations} to the proposition below).
\end{remark}

\begin{remark}
The proof of \Cref{thm:moment_formula} relies on machinery from symmetric function theory, where zonal spherical functions provide the link between symmetric functions and products of random matrices, see \cite{ahn2019fluctuations} for details.
\end{remark}

We want to consider the general setting where some the $L_\tau^{(n)}$ may be infinite. To avoid separating cases between $L_\tau^{(n)}$ finite and infinite, it will be convenient to shift the log squared singular values by an additive factor defined below.

\begin{definition}
For positive integers $L > n$ we let
\[
s_n(L) := \sum_{k=1}^{L-n} \frac{1}{k} - \log L = H_{L-n} - \log L
\]
and let
\[ s_n(\infty) = \lim_{L \to \infty} s_n(L) = \gamma \]
where $\gamma$ is the Euler-Mascheroni constant.
\end{definition}

\begin{corollary}\label{thm:moment_formula_infinite}
If $c_1,\ldots,c_m > 0$ are real, and $L_1^{(n)},\ldots, L_T^{(n)} \in \Z_{>n} \cup \{\infty\}$, then
\begin{align}\label{eq:moment_formula_infinite}
\begin{split}
\begin{multlined}
\E\left[ \prod_{i=1}^m \sum_{j=1}^n \left( y_j^{(n)}(T) e^{\sum_{\tau=1}^T s_n(L_\tau^{(n)})} \right)^{c_i} \right] = \frac{\prod_{i=1}^m (-c_i)^{-1}}{(2\pi\bi)^m} \oint \cdots \oint \prod_{1 \le i < j \le m} \frac{(u_j - u_i)(u_j + c_j - u_i - c_i)}{(u_j - u_i - c_i)(u_j + c_j - u_i)} \\
\times \prod_{i=1}^m \left( \prod_{\ell=1}^n \frac{u_i + \ell - 1}{u_i + c_i + \ell - 1} \cdot \prod_{\tau=1}^T \prod_{k=1}^{L_\tau^{(n)} - n} e^{c_i/k} \frac{u_i + c_i - k}{u_i - k} \right) du_i
\end{multlined}
\end{split}
\end{align}
where the $u_i$-contour $\fU_i$ is positively oriented around $\{-c_i- \ell+1\}_{\ell=1}^n$ but does not enclose $k$ for $1 \le k \le \max_{1 \le \tau \le T} L_\tau^{(n)} - n$, and is enclosed by $\fU_j - c_i, \fU_j + c_j$ for $j > i$; the result holds provided that such contours exist.
\end{corollary}

\begin{proof}
First suppose that all $L_\tau^{(n)}$ are finite. Then the result follows from \Cref{thm:moment_formula}, where we note that the shifts $s_n(L_\tau^{(n)})$ effectively replace the $(L_\tau^{(n)})^{c_i}$ factors inside the product over $\tau$ in the RHS of \eqref{eq:moment_formula} with $\prod_{k=1}^{L_\tau^{(n)} - n} e^{c_i/k}$. For the general case, we have that for $u \notin \{1,\ldots,\max_{1 \le \tau \le T} L_\tau^{(n)} - n\}$ and $c > 0$,
\[\prod_{k=1}^{L_\tau^{(n)} - n} e^{c/k} \frac{u + c - k}{u - k} \to \prod_{k=1}^{\infty} e^{c/k} \frac{u + c - k}{u - k}\]
as $L_\tau^{(n)} \to \infty$. Additionally, $X_\tau^{(n)}$ converges weakly to an $n \times n$ Ginibre matrix as $L_\tau^{(n)} \to \infty$ by \cite{petz2004asymptotics}. Hence the general case follows by taking a limit of both sides of the finite case.
\end{proof}

\begin{lemma}\label{thm:tau_product_converges}
Fix $n \in \Z_{>0}$, $\wh{c} \in \R_{>0}$, and let $(L_\tau^{(n)})_{\tau \ge 1}$ and $\rho_n$ be as in \Cref{thm:concentration}. Then
\[ \lim_{T\to\infty} \prod_{\tau=1}^T \prod_{k=1}^{L_\tau^{(n)} - n} e^{\frac{\wh{c}}{kT}} \frac{u+\frac{\wh{c}}{T}-k}{u-k} = \prod_{k=1}^\infty \exp\left( \wh{c} \,\rho_n(\llbracket k,\infty \rrbracket)\left( \frac{1}{k} + \frac{1}{u - k} \right) \right) \]
uniformly over $u$ in compact subsets of $\C \setminus \Z_{>0}$.
\end{lemma}

\begin{proof}
We can write
\begin{align*}
\prod_{\tau=1}^T \prod_{k=1}^{L_\tau^{(n)} - n} e^{\frac{\wh{c}}{kT}} \frac{u+\frac{\wh{c}}{T}-k}{u-k} &= \prod_{\tau=1}^T \prod_{k=1}^{L_\tau^{(n)} - n} e^{\frac{\wh{c}}{kT}} \left( 1 - \frac{1}{T} \frac{\wh{c}}{k - u} \right) \\
&= \prod_{k=1}^\infty \left( e^{\frac{\wh{c}}{kT}} \left( 1 - \frac{1}{T} \frac{\wh{c}}{k - u} \right) \right)^{\#\{1 \le \tau \le T: L_\tau^{(n)} - n \ge k\}}.
\end{align*}
Note that the product is always convergent. Indeed, observe that
\[ \frac{C^{-1}}{k^2 T} \le \log \left( e^{\frac{\wh{c}}{kT}} \left( 1 - \frac{1}{T} \frac{\wh{c}}{k - u} \right) \right) \le \frac{C}{k^2 T} \]
for $u$ in a compact subset of $\C \setminus \Z_{>0}$, for $k$ sufficiently large and for some $C>1$ depending on this subset. Since $\#\{1 \le \tau \le T: L_\tau^{(n)} - n \ge k\} = \rho_n(\llbracket k,\infty \rrbracket) T (1 + o(1))$ the desired convergence follows. 
\end{proof}

\begin{proof}[Proof of {\Cref{thm:concentration}}]

As we will show below,
\begin{align}\label{eq:observe_that}
\lim_{T \to \infty} \frac{1}{T} \sum_{\tau=1}^T s_n(L_\tau^{(n)}) = - \log(n) + \sum_{k=1}^\infty \rho_n(\llbracket k,\infty \rrbracket) \left( \frac{1}{k} + \log\left(1 - \frac{1}{k+n} \right) \right) =: \alpha.
\end{align}
Thus, we may define
\[ \mu_T := \frac{1}{n} \sum_{i=1}^n \delta_{x_i(T)}, \quad \quad \mu := \frac{1}{n} \sum_{i=1}^n \delta_{\lambda_i(n) + \alpha} \]
where
\[ x_i(T) = \frac{1}{T} \left( \log y_i^{(n)}(T) + \sum_{\tau=1}^T s_n(L_\tau^{(n)}) \right), \quad \quad 1 \le i \le n. \]

We establish our theorem by showing that
\begin{align} \label{eq:laplace_transform}
\int e^{\wh{c}x} d\mu_T(x) \to \int e^{\wh{c}x} d\mu(x)
\end{align}
in probability as $T \to \infty$, for any $\wh{c} > 0$. In fact, it suffices to show this for $\wh{c} = 1,\ldots,n$ as this would imply that
\[ p_k(e^{x_1(T)},\ldots,e^{x_n(T)}) \to p_k(e^{\lambda_1(n) + \alpha},\ldots,e^{\lambda_n(n) + \alpha}) \]
in probability $T \to\infty$ for each $k = 1,\ldots,n$. Since the first $n$ power sums algebraically generate any symmetric polynomial in $n$ variables, we have
\[ e_k(e^{x_1(T)},\ldots,e^{x_n(T)}) \to e_k(e^{\lambda_1(n) + \alpha},\ldots,e^{\lambda_n(n) + \alpha}) \]
in probability as $T \to\infty$ for each $k = 1,\ldots,n$. As
\[ (z - r_1) \cdots (z - r_n) = z^n - e_1(r_1,\ldots,r_n)z^{n-1} + \cdots + (-1)^n e_n(r_1,\ldots,r_n) \]
and the roots of a polynomial depend continuously on its coefficients, we see that
\[ e^{x_i(T)} \to e^{\lambda_i(n)+\alpha} \]
and therefore
\[ \frac{1}{T}\log y_i^{(n)}(T) \to \lambda_i(n) \]
in probability as $T\to\infty$ for $1 \le i \le n$, as desired.

The remainder of the proof has three parts: establishing \eqref{eq:observe_that}, showing the convergence in expectation of Laplace transforms \eqref{eq:laplace_transform}, and showing the variance of the LHS is $o(1)$ to upgrade to convergence in probability. 

\vspace{5mm}
\noindent
\textbf{Part 1: Establishing \eqref{eq:observe_that}.}

Define
\[ \rho_{n,T}(k) = \frac{1}{T} \#\{L_\tau^{(n)} - n = k, 1 \le \tau \le T \}. \]
In particular, 
\[ \lim_{T \to \infty} \rho_{n,T}(\llbracket k,\infty \rrbracket) = \rho_n(\llbracket k,\infty \rrbracket). \]
Then
\begin{align*}
\frac{1}{T} \sum_{\tau=1}^T s_n(L_\tau^{(n)}) &= \sum_{k \in \Z_{> 0} \cup \{\infty\}} \rho_{n,T}(k) s_n(k + n) \\
&= \rho_{n,T}(\infty) \gamma + \sum_{k \in \Z_{\ge0}} \rho_{n,T}(k) \left( H_k - \log(k + n) \right)
\end{align*}
where the last line uses the fact that $\rho_{n,T}(0) = 0$ (recall $L_\tau^{(n)} > n$). Using summation by parts
\[ \sum_{k=0}^m f_k(g_{k+1} - g_k) = (f_mg_{m+1} - f_0 g_0) - \sum_{k=1}^m g_k(f_k - f_{k-1}), \]
we have
\begin{align*}
& \sum_{k=0}^m \rho_{n,T}(k) \left( H_k - \log(k+n) \right) \\
& \quad = (\log(m+n) - H_m) \rho_{n,T}(\llbracket m+1,\infty \rrbracket) - \log(n) \rho_{n,T}(\llbracket 0,\infty \rrbracket) \\
& \quad \quad \quad \quad + \sum_{k=1}^m \rho_{n,T}(\llbracket k,\infty \rrbracket) \left(H_k - H_{k-1} + \log\left(\frac{k+n-1}{k+n}\right) \right) \\
& \quad = (\log(m+n) - H_m) \rho_{n,T}(\llbracket m+1,\infty \rrbracket) - \log(n) + \sum_{k=1}^m \rho_{n,T}(\llbracket k,\infty \rrbracket) \left(\frac{1}{k} + \log\left(1 - \frac{1}{k+n}\right) \right).
\end{align*}
Sending $m$ to $\infty$, we obtain
\[ -\rho_{n,T}(\infty) \gamma - \log(n) + \sum_{k=1}^\infty \rho_{n,T}(\llbracket k,\infty \rrbracket) \left( \frac{1}{k} + \log\left(1 - \frac{1}{k+n}\right) \right). \]
Note that the latter series is convergent. Therefore
\[ \frac{1}{T} \sum_{\tau=1}^T s_n(L_\tau^{(n)}) = - \log(n) + \sum_{k=1}^\infty \rho_{n,T}(\llbracket k,\infty \rrbracket) \left( \frac{1}{k} + \log\left(1 - \frac{1}{k+n}\right) \right). \]
Taking $T$ to $\infty$ establishes \eqref{eq:observe_that}.

\vspace{5mm}
\noindent
\textbf{Part 2: Convergence in expectation.}

By \Cref{thm:moment_formula_infinite}, we have
\begin{align*}
\E \left[ \int e^{Tcx} d\mu_T(x) \right] &= \E\left[ \sum_{i=1}^n \left(y_i^{(n)}(T) e^{\sum_{\tau=1}^T s_n(L_\tau^{(n)})} \right)^c \right] \\ 
&= \frac{-c^{-1}}{2\pi\bi} \oint \prod_{\ell=1}^n \frac{u + \ell - 1}{u + c + \ell - 1} \cdot \prod_{\tau=1}^T \prod_{k=1}^{L_\tau^{(n)} - n} e^{c/k} \frac{u + c - k}{u - k} du \\
&= \sum_{\ell=1}^n \left( \prod_{h \ne \ell} \frac{h-\ell-c}{h-\ell} \right) \prod_{\tau=1}^T \prod_{k=1}^{L_\tau^{(n)} - n} e^{c/k} \frac{-\ell+1-k}{-c-\ell+1-k}
\end{align*}
where we have expanded the integral in its residues at $-c, -c-1,\ldots,-c-n+1$. Setting $c = \wh{c}/T$ and applying \Cref{thm:tau_product_converges} to the product over $\tau$, we obtain
\begin{align*}
\lim_{T\to\infty} \E \left[ \int e^{\wh{c}x} d\mu_T(x) \right] &= \sum_{\ell=1}^n \prod_{k=1}^\infty \exp\left( \wh{c}\,\rho_n(\llbracket k,\infty \rrbracket)\left( \frac{1}{k} - \frac{1}{k+\ell - 1} \right) \right) \\
&= \sum_{i=1}^n \exp\left( \wh{c} \sum_{k=1}^\infty \rho_n(\llbracket k,\infty \rrbracket)\left( \frac{1}{k} - \frac{1}{k+n - i} \right) \right)
\end{align*}
where in the second equality we change indices $i = n - \ell + 1$. The convergence 
\[ \lim_{T\to\infty} \E \left[ \int e^{\wh{c}x} d\mu_T(x) \right] = \sum_{i=1}^n e^{\wh{c}(\lambda_i(n) + \alpha)} \]
follows from the fact that
\[ \lambda_i(n) + \alpha = \sum_{k=1}^\infty \rho_n(\llbracket k,\infty \rrbracket)\left( \frac{1}{k} - \frac{1}{k+n - i} \right) \]
where we recall the definition of $\lambda_i(n)$ in \eqref{eq:main_conv_in_prob}.

\vspace{5mm}
\noindent
\textbf{Part 3: Vanishing variance.} 

By \Cref{thm:moment_formula_infinite} we have
\begin{align*}
    & \var\left( \sum_{i=1}^n \left(y_i^{(n)}(T)e^{\sum_{\tau=1}^T s_n(L_\tau^{(n)})}\right)^c \right)\\
    &= \E\left[\left( \sum_{i=1}^n \left(y_i^{(n)}(T)e^{\sum_{\tau=1}^T s_n(L_\tau^{(n)})}\right)^c \right)^2\right] - \E\left[ \sum_{i=1}^n \left(y_i^{(n)}(T)e^{\sum_{\tau=1}^T s_n(L_\tau^{(n)})}\right)^c\right]^2 \\
& =   \frac{c^{-2}}{(2\pi\bi)^2} \oint \oint \left( \frac{(u_2 - u_1)^2}{(u_2 - u_1 - c)(u_2 - u_1 + c)} - 1 \right) \prod_{i=1}^2 \left( \prod_{\ell=1}^n \frac{u_i + \ell - 1}{u_i + c + \ell - 1} \cdot \prod_{\tau=1}^T \prod_{k=1}^{L_\tau^{(n)} - n} e^{c/k} \frac{u_i + c - k}{u_i - k} \right) du_i \\
& = \frac{1}{(2\pi\bi)^2} \oint \oint \frac{1}{(u_2 - u_1)^2 - c^2} \prod_{i=1}^2 \left( \prod_{\ell=1}^n \frac{u_i + \ell - 1}{u_i + c + \ell - 1} \cdot \prod_{\tau=1}^T \prod_{k=1}^{L_\tau^{(n)} - n} e^{c/k} \frac{u_i + c - k}{u_i - k} \right) du_i.\\
\end{align*}
Since the integration is over closed contours which do not depend on $T$ (at least for $T$ large enough), letting $c=\wh{c}/T$ we have by \Cref{thm:tau_product_converges} that the above converges as $T \to \infty$ to 
\[
\frac{1}{(2\pi\bi)^2} \oint \oint \frac{g(u_1) g(u_2)}{(u_2 - u_1)^2} \, du_1 \, du_2
\]
where $g$ is analytic on $\C \setminus \Z_{>0}$. Since the contours were chosen so that the $u_2$ contour encloses the $u_1$ contour and does not enclose any $k \in \Z_{>0}$, the integral with respect to $u_1$ vanishes. Therefore
\[
\lim_{T \to \infty} \var\left( \sum_{i=1}^n \left(y_i^{(n)}(T)e^{\sum_{\tau=1}^T s_n(L_\tau^{(n)})}\right)^{\wh{c}/T} \right) = 0.
\]
The convergence in probability of Laplace transforms now follows from convergence in expectation and Chebyshev's inequality by the standard argument.

\end{proof}

\begin{proof}[Proof of \Cref{thm:lyapunov}]
By the definition \eqref{eq:main_conv_in_prob} of $\lambda_i(n)$, 
\begin{align} \label{eq:lambda_diff}
\begin{split}
\lambda_i(n) - \lambda_1(n) &= \sum_{k = 1}^\infty \rho_n(\llbracket k,\infty \rrbracket)\left(\frac{1}{k-n-1} - \frac{1}{k-n-i}\right) \\
&= -(i-1) c(n) + \e_i(n)
\end{split}
\end{align}
where $c(n)$ is as defined in the theorem statement, and
\[ \e_i(n) := -(i-1)^2 \sum_{k=1}^\infty \frac{\rho_n(\llbracket k,\infty \rrbracket)}{(k+n-1)^2(k+n-i)}. \]
Notice that
\[ |\e_i(n)| \le \frac{i-1}{n-i+1} (i-1) \sum_{k=1}^\infty \frac{\rho_n(\llbracket k,\infty \rrbracket)}{(k+n-1)^2} = \frac{(i-1)^2}{n-i+1} c(n). \]
Thus for $m(n) = o(\sqrt{n})$ it is clear that 
\[
\sup_{1 \leq i \leq m(n)} |\e_i(n)| = |\e_{m(n)}(n)| = c(n) o(1),
\]
completing the proof.
\end{proof}

\bibliographystyle{alpha}
\bibliography{references}

\end{document}

%% file: preamble.tex
\usepackage[utf8]{inputenc}
\usepackage{amsthm,amsmath,amsfonts,amssymb}
\usepackage{amsrefs,hyperref,cleveref}
\usepackage{multirow}
\usepackage{graphicx}
\usepackage{mathtools}
\usepackage{enumerate}
\usepackage[myheadings]{fullpage}
\usepackage{comment}
\usepackage{stmaryrd}

\newtheorem{theorem}{Theorem}[section]

\newtheorem{lemma}[theorem]{Lemma}
\newtheorem{proposition}[theorem]{Proposition}
\newtheorem{corollary}[theorem]{Corollary}

\theoremstyle{definition}
\newtheorem{definition}{Definition}

\newtheorem{example}{Example}[section]
\newtheorem{remark}{Remark}

\newcommand{\R}{\mathbb{R}}
\newcommand{\C}{\mathbb{C}}

\newcommand{\E}{\mathbb{E}}
\newcommand{\N}{\mathbb{N}}

\renewcommand{\P}{\mathbb{P}}
\newcommand{\Z}{\mathbb{Z}}

\newcommand{\Tr}{\mathrm{Tr}\,}
\renewcommand{\vec}[1]{\boldsymbol{#1}}

\newcommand{\var}{\mathrm{Var}}
\newcommand{\e}{\varepsilon}

\newcommand{\bi}{\mathbf{i}}

\newcommand{\fU}{\mathfrak{U}}

\newcommand{\wh}[1]{\widehat{#1}}
\newcommand{\wt}[1]{\widetilde{#1}}
\renewcommand{\vec}[1]{\boldsymbol{#1}}

\newcommand{\GL}{\mathrm{GL}}

\newcommand{\tth}{^{th}}

\numberwithin{equation}{section}

%% file: main.bbl
\begin{bibdiv}
\begin{biblist}

\bib{aggarwal2019universality}{article}{
      author={Aggarwal, Amol},
       title={Universality for lozenge tiling local statistics},
        date={2019},
     journal={arXiv preprint arXiv:1907.09991},
}

\bib{aggarwal2021gaussian}{article}{
      author={Aggarwal, Amol},
      author={Gorin, Vadim},
       title={Gaussian unitary ensemble in random lozenge tilings},
        date={2021},
     journal={arXiv preprint arXiv:2106.07589},
}

\bib{aggarwal2021edge}{article}{
      author={Aggarwal, Amol},
      author={Huang, Jiaoyang},
       title={{Edge Statistics for Lozenge Tilings of Polygons, II: Airy Line
  Ensemble}},
        date={2021},
     journal={arXiv preprint arXiv:2108.12874},
}

\bib{ahn2019fluctuations}{article}{
      author={Ahn, Andrew},
       title={Fluctuations of $\beta$-{J}acobi product processes},
     journal={to appear in Probability Theory and Related Fields,
  arXiv:1910.00743},
}

\bib{ahn2021unpublished}{article}{
      author={Ahn, Andrew},
       title={Extremal singular values of random matrix products and {B}rownian
  motion on {$\mathrm{GL}(N,\mathbb{C})$}},
        date={2022},
     journal={arXiv preprint arXiv:2201.11809},
}

\bib{akemann2019integrable}{article}{
      author={Akemann, Gernot},
      author={Burda, Zdzislaw},
      author={Kieburg, Mario},
       title={From integrable to chaotic systems: {U}niversal local statistics
  of {L}yapunov exponents},
        date={2019},
     journal={EPL (Europhysics Letters)},
      volume={126},
      number={4},
       pages={40001},
}

\bib{akemann2020universality}{article}{
      author={Akemann, Gernot},
      author={Burda, Zdzislaw},
      author={Kieburg, Mario},
       title={Universality of local spectral statistics of products of random
  matrices},
        date={2020},
     journal={arXiv preprint arXiv:2008.11470},
}

\bib{bellman1954limit}{article}{
      author={Bellman, Richard},
       title={Limit theorems for non-commutative operations. {I}.},
        date={1954},
     journal={Duke Mathematical Journal},
      volume={21},
      number={3},
       pages={491\ndash 500},
}

\bib{borodin2014macdonald}{article}{
      author={Borodin, Alexei},
      author={Corwin, Ivan},
       title={Macdonald processes},
        date={2014},
     journal={Probability Theory and Related Fields},
      volume={158},
      number={1-2},
       pages={225\ndash 400},
}

\bib{borodin2015general}{article}{
      author={Borodin, Alexei},
      author={Gorin, Vadim},
       title={{General $\beta$-{J}acobi Corners Process and the Gaussian Free
  Field}},
        date={2015},
     journal={Communications on Pure and Applied Mathematics},
      volume={68},
      number={10},
       pages={1774\ndash 1844},
}

\bib{borodin2018product}{article}{
      author={Borodin, Alexei},
      author={Gorin, Vadim},
      author={Strahov, Eugene},
       title={Product matrix processes as limits of random plane partitions},
        date={2018},
     journal={International Mathematics Research Notices},
}

\bib{crisanti1993products}{book}{
      author={Crisanti, A.},
      author={Paladin, G.},
      author={Vulpiani, A.},
       title={Products of random matrices in statistical physics},
      series={Springer Series in Solid-State Sciences},
   publisher={Springer-Verlag, Berlin},
        date={1993},
      volume={104},
        ISBN={3-540-56575-2},
         url={https://doi-org.libproxy.mit.edu/10.1007/978-3-642-84942-8},
        note={With a foreword by Giorgio Parisi},
      review={\MR{1278483}},
}

\bib{erdHos2012universality}{article}{
      author={Erd{\H{o}}s, L{\'a}szl{\'o}},
      author={Yau, Horng-Tzer},
       title={Universality of local spectral statistics of random matrices},
        date={2012},
     journal={Bulletin of the American Mathematical Society},
      volume={49},
      number={3},
       pages={377\ndash 414},
}

\bib{forrester2013lyapunov}{article}{
      author={Forrester, Peter~J},
       title={Lyapunov exponents for products of complex {G}aussian random
  matrices},
        date={2013},
     journal={Journal of Statistical Physics},
      volume={151},
      number={5},
       pages={796\ndash 808},
}

\bib{furstenberg1960products}{article}{
      author={Furstenberg, Harry},
      author={Kesten, Harry},
       title={Products of random matrices},
        date={1960},
     journal={The Annals of Mathematical Statistics},
      volume={31},
      number={2},
       pages={457\ndash 469},
}

\bib{gorin2020crystallization}{article}{
      author={Gorin, Vadim},
      author={Marcus, Adam~W},
       title={Crystallization of random matrix orbits},
        date={2020},
     journal={International Mathematics Research Notices},
      volume={2020},
      number={3},
       pages={883\ndash 913},
}

\bib{gorin2018gaussian}{article}{
      author={Gorin, Vadim},
      author={Sun, Yi},
       title={Gaussian fluctuations for products of random matrices},
        date={2018},
     journal={arXiv preprint arXiv:1812.06532},
}

\bib{grabiner1999brownian}{inproceedings}{
      author={Grabiner, David~J},
       title={Brownian motion in a {W}eyl chamber, non-colliding particles, and
  random matrices},
        date={1999},
   booktitle={{Annales de l'Institut Henri Poincar{\'e}, Probabilit{\'e}s et
  Statistiques}},
      volume={35},
       pages={177\ndash 204},
}

\bib{hanin2020products}{article}{
      author={Hanin, Boris},
      author={Nica, Mihai},
       title={Products of many large random matrices and gradients in deep
  neural networks},
        date={2020},
        ISSN={0010-3616},
     journal={Comm. Math. Phys.},
      volume={376},
      number={1},
       pages={287\ndash 322},
         url={https://doi-org.libproxy.mit.edu/10.1007/s00220-019-03624-z},
      review={\MR{4093863}},
}

\bib{hsu2002stochastic}{book}{
      author={Hsu, Elton~P},
       title={Stochastic analysis on manifolds},
   publisher={American Mathematical Soc.},
        date={2002},
      number={38},
}

\bib{ipsen2016isotropic}{article}{
      author={Ipsen, Jepser~R},
      author={Schomerus, H},
       title={Isotropic {B}rownian motions over complex fields as a solvable
  model for {M}ay--{W}igner stability analysis},
        date={2016},
     journal={Journal of Physics A: Mathematical and Theoretical},
      volume={49},
      number={38},
       pages={385201},
}

\bib{isopi1992triangle}{article}{
      author={Isopi, Marco},
      author={Newman, Charles~M},
       title={The triangle law for {L}yapunov exponents of large random
  matrices},
        date={1992},
     journal={Communications in mathematical physics},
      volume={143},
      number={3},
       pages={591\ndash 598},
}

\bib{jones2006weyl}{article}{
      author={Jones, Liza},
      author={O’Connell, Neil},
       title={Weyl chambers, symmetric spaces and number variance saturation},
        date={2006},
     journal={ALEA Lat. Am. J. Probab. Math. Stat},
      volume={2},
       pages={91\ndash 118},
}

\bib{liu2018lyapunov}{article}{
      author={Liu, Dang-Zheng},
      author={Wang, Dong},
      author={Wang, Yanhui},
       title={Lyapunov exponent, universality and phase transition for products
  of random matrices},
        date={2018},
     journal={arXiv preprint arXiv:1810.00433},
}

\bib{newman1986distribution}{article}{
      author={Newman, Charles~M},
       title={The distribution of {L}yapunov exponents: Exact results for
  random matrices},
        date={1986},
     journal={Communications in mathematical physics},
      volume={103},
      number={1},
       pages={121\ndash 126},
}

\bib{newman1986lyapunov}{article}{
      author={Newman, Charles~M},
       title={Lyapunov exponents for some products of random matrices: exact
  expressions and asymptotic distributions},
        date={1986},
     journal={Random Matrices and Their Applications (Contemporary Mathematics
  50), American Mathematical Society, Providence},
       pages={121\ndash 141},
}

\bib{oseledets1968multiplicative}{article}{
      author={Oseledets, Valery~Iustinovich},
       title={A multiplicative ergodic theorem. {C}haracteristic {L}japunov
  exponents of dynamical systems},
        date={1968},
     journal={Trudy Moskovskogo Matematicheskogo Obshchestva},
      volume={19},
       pages={179\ndash 210},
}

\bib{petz2004asymptotics}{article}{
      author={Petz, D\'{e}nes},
      author={R\'{e}ffy, J\'{u}lia},
       title={On asymptotics of large {H}aar distributed unitary matrices},
        date={2004},
        ISSN={0031-5303},
     journal={Period. Math. Hungar.},
      volume={49},
      number={1},
       pages={103\ndash 117},
  url={https://doi-org.libproxy.mit.edu/10.1023/B:MAHU.0000040542.56072.ab},
      review={\MR{2092786}},
}

\bib{raghunathan1979proof}{article}{
      author={Raghunathan, Madabusi~S},
       title={A proof of {O}seledec’s multiplicative ergodic theorem},
        date={1979},
     journal={Israel Journal of Mathematics},
      volume={32},
      number={4},
       pages={356\ndash 362},
}

\bib{van2020limits}{article}{
      author={Van~Peski, Roger},
       title={Limits and fluctuations of p-adic random matrix products},
        date={2021},
     journal={Selecta Mathematica},
      volume={27},
      number={5},
       pages={1\ndash 71},
}

\end{biblist}
\end{bibdiv}
